\documentclass[12pt]{amsart}

\usepackage{amsmath}
\usepackage{amssymb}
\usepackage{amscd}
\usepackage{a4wide}
\usepackage[utf8]{inputenc}
\newtheorem{thm}{Theorem}[section]
\newtheorem{cor}[thm]{Corollary}
\newtheorem{lem}[thm]{Lemma}
\newtheorem{prop}[thm]{Proposition}
\newcommand{\g}{\mathfrak{g}}

\newcommand{\m}{\mathfrak{m}}
\newcommand{\CC}{\mathbb{C}}

\newcommand{\ZZ}{\mathbb{Z}}

\newcommand{\pp}{\mathfrak{p}}

    \usepackage{color}
    
\begin{document}

\title{Injective hulls of simple modules over differential operator rings}

\author{Paula A.A.B. Carvalho, Can Hatipo\u{g}lu, Christian Lomp}

\address{Department of Mathematics, Faculty of Science, University of Porto, Rua Campo Alegre 687, 4169-007 Porto, Portugal}

\begin{abstract}
We study injective hulls of simple modules over differential operator rings $R[\theta; d]$, providing necessary conditions under which these modules are locally Artinian. As a consequence we characterize Ore extensions of 
$S=K[x][\theta;\sigma, d]$ for $d$ a $K$-linear derivation and $\sigma$ a $K$-linear automorphism of $K[x]$ such that 
injective hulls of simple $S$-modules are locally Artinian. 
\end{abstract}

\maketitle

\section{Introduction}
In 1960, E.Matlis showed that any injective hull of a simple module over a commutative Noetherian ring is Artinian (see 
\cite{Matlis} and \cite[Proposition 3]{Matlis_DCC}).
In connection with the Jacobson Conjecture for Noetherian rings, \textit{i.e.} whether $\bigcap_{n=1}^\infty J^n = 0$ for the 
Jacobson radical $J$ of a Noetherian ring $R$,
Jategaonkar showed in \cite{Jategaonkar_conj} (see also \cite{Cauchon}, \cite{Schelter}) that the injective hulls of 
simple modules are locally Artinian provided the ring $R$ is fully bounded Noetherian (FBN). This allowed him to answer 
the Jacobson Conjecture in the affirmative for FBN rings. Recall that a module is called {\it locally Artinian} if every 
finitely generated submodule of it is Artinian.
After Jategaonkar's result the question arose whether for a given Noetherian ring $A$ the condition
\begin{center}$(\diamond)\:\:$ Injective hulls of simple left $A$-modules are locally Artinian\end{center}
was necessary to prove Jacobson Conjecture which quickly turned out not to be the case. 
Property $(\diamond)$ says that all finitely generated essential extensions of simple left $A$-modules are Artinian and 
in case 
$A$ is left Noetherian the property above is equivalent to the condition that the class of semi-Artinian left 
$A$-modules, \textit{i.e.} modules $M$ that  are the union of their socle series, is closed under essential extensions (see 
\cite{Can}). However property $(\diamond)$ remained a subtle condition for Noetherian rings whose meaning is not yet 
fully understood. 

For algebras related to $U(\mathfrak{sl}_2)$ the condition has been examined 
in \cite{Dahlberg},\cite{injective_modules_over_down-up_algebras}, \cite{PaulaIan}, \cite{musson_classification}. 
One of the first examples of a Noetherian domain that does not satisfy $(\diamond)$ has been found by Ian Musson 
in \cite{some-examples-of-modules-over-noetherian-rings-musson} concluding that whenever  $\g$ is a finite dimensional 
solvable non-nilpotent Lie algebra, 
then $U(\g)$ does not satisfy property $(\diamond)$. The class of finite dimensional nilpotent Lie algebras $\g$ for 
which $U(\g)$ satisfies our desired property has 
been recently determined in \cite{CanChristian}. The main result of this paper is a complete and explicit answer to the 
question when $S=K[x][\theta;d]$
satisfies property $(\diamond)$, for $K$ a field of characteristic zero and $d$ a $K$-linear derivation of $K[x]$. We show that there always exists a non-Artinian 
finitely generated essential 
extension of a simple $S$-module if and only if  $d$  is not locally nilpotent or equivalently if and only if 
$S$ is neither isomorphic to the polynomial ring $K[x,y]$ nor to the first Weyl algebra $A_1(K)$. This extends a recent 
result by I.Musson for the derivations 
$d(x)=x^r$ (see \cite{musson_classification}).  Applying a result of P.Carvalho and I.Musson \cite{PaulaIan} and results 
of Alev \textit{et al.} \cite{AlevDumas} and Awami \textit{et al.} \cite{AwamiVandenBerghVanOystaeyen}, we can also characterise those Ore extensions of $K[x]$ which satisfy  $(\diamond)$.

Let $R$ be a ring, $d$ a derivation of $R$, \textit{i.e.} an additive map $d:R\rightarrow R$ such that $d(ab)=d(a)b+ad(b)$ for all $a,b\in R$. 
 A subset $I$ of $R$ is called $d$-stable, if $d(I)\subseteq I$. An ideal of 
$R$ that is $d$-stable is called a $d$-ideal. The set $R^d=\{a\in R \mid d(a)=0\}$ is a subring of $R$  called the 
\emph{subring of constants of $d$}. 

The \emph{differential operator ring} of $R$ with respect to $d$ is 
the ring $S=R[\theta;d]$, \textit{i.e.} $S$ is an overring of $R$, free as left $R$-module with basis $\{\theta^n \mid n\geq 0\}$ 
subject to the relation
$ \theta a = a\theta + d(a)$, for all $a\in R$.
Moreover  $S$ is also free as right $R$-module with basis $\{\theta^n\mid n\geq 0\}$ and the following identities hold 
(see \cite{GoodearlWarfield}):
$$\displaystyle \theta^n a =\sum_{i=0}^n {n \choose i} d^{n-i}(a) \theta^i \qquad \mbox{ and } \qquad 
\displaystyle a\theta^n =\sum_{i=0}^n {n \choose i}(-1)^i \theta^{n-i}d^{i}(a)  \qquad \forall a\in R, n\geq 0.$$

The ring $R$ becomes a left $S$-module by letting  $\theta$ act as $d$ on $R$, \textit{i.e.} 
$a\theta^n \cdot x = ad^n(x)$ for all $a,x\in R$ and $n\geq 0$. 
The left $S$-submodules of $R$ are precisely the left ideals of $R$ that are $d$-stable and it is not difficult to prove 
that $R^d$ is isomorphic to the ring of endomorphisms of $R$ as left $S$-module.  The map $\alpha: S\longrightarrow R$ with  
$\alpha\left(\sum_{i=0}^n a_i\theta^i\right) = \sum_{i=0}^n a_i\theta^i\cdot 1 = a_0,$ for all 
$\sum_{i=0}^n a_i\theta^i  \in S$ is an epimorphism of left $S$-modules and splits as left $R$-modules, by the inclusion map $R\subseteq S$. Since 
the kernel of $\alpha$ is $S\theta$, we have $S/S\theta \simeq R$ as left $S$-module. Hence in particular the lattice 
$\mathcal{L}({_S{S/S\theta}})$ of left $S$-submodules of $S/S\theta$ is isomorphic to the lattice $\mathcal{L}({_SR})$ 
of $d$-stable left ideals of $R$.

%%%%%%% Motivation and outline of the paper%%%%%

In this paper we discuss necessary and sufficient conditions for $S=R[\theta; d]$ to satisfy property $(\diamond)$.  
Given a commutative Noetherian domain $R$ which is a finitely generated $K$-algebra over an algebraically closed  field $K$ of  characteristic $0$, and a $K$-linear locally nilpotent derivation $d$, we show in the second section, that $S=R[\theta; d]$ satisfies property $(\diamond)$  (see Proposition \ref{prop_positive}). However if $R$ is not commutative, we show that $S$ might not satisfy property $(\diamond)$  (see Theorem \ref{lnd_artinian}).

In the last section we show that if $R$ is a commutative Noetherian domain such that  $S$ satisfies  $(\diamond)$, $R$ being $d$-primitive implies $R$ to be $d$-simple.  I.Musson's first example of a Noetherian domain not satisfying $(\diamond)$ was the ring $K[x][\theta; x\frac{\partial}{\partial x}]$ where $K$ is a field (see \cite{some-examples-of-modules-over-noetherian-rings-musson}). In \cite{musson_classification} he extended his construction of a cyclic non-Artinian essential extension of a simple  module to rings of the form $K[x][\theta; x^r\frac{\partial}{\partial x}]$ with $r\geq 1$.
In this paper, extending I.Musson's results,  we classify completely the differential operator rings  $S=K[x][\theta; d]$ satisfying $(\diamond)$, with $d$ a $K$-linear derivation,  as those with $d$ being locally nilpotent, \textit{i.e.} $S$ is either commutative or isomorphic to the first Weyl algebra. As a Corollary a complete characterisation of Ore extensions  $K[x][\theta;\alpha,d]$  with property $(\diamond)$  is obtained in Theorem \ref{OreExtension}.

%%%%%%% Motivation and outline of the paper%%%%%

\section{Locally Nilpotent Derivations}
In this section we discuss locally nilpotent derivations $d$ on a ring $R$.  Recall that a derivation $d$ of $R$ is called \emph{locally nilpotent} if for every $a\in R$ there exists $n>0$ such that  $d^n(a)=0$. Note that if $d$ is locally nilpotent, then every non-zero $d$-stable left ideal intersects non-trivially the subring of constants $R^d$. The first positive result shows that over a commutative finitely generated $K$-algebra $R$ the differential 
operator ring $S=R[\theta;d]$ satisfies $(\diamond)$ provided $d$ is locally nilpotent.

\begin{prop}\label{prop_positive} Let $K$ be an algebraically closed field of characteristic zero and let $R$ be a 
commutative finitely generated $K$-algebra with locally nilpotent derivation $d$. Then $R[\theta; d]$ satisfies property  $(\diamond)$.
\end{prop}

\begin{proof}
Let $x_1, \ldots, x_n$ be the algebra generators of $R$ and consider the set $$V=\left\{ d^i(x_j) \mid 1\leq j \leq n, 
i\geq 0\right\}.$$ Since $d$ is locally nilpotent, $V$ is a finite set containing all generators $x_1, \ldots, x_n$. Let 
$\mathfrak{h}=\mathrm{span}(V)$ be the (finite dimensional) subspace of $R$ generated by $V$. Consider 
$\mathfrak{g}=\mathfrak{h}\oplus K\theta$, which is a  subspace of $S=R[\theta; d]$. Since $[\theta, d^i(x_j)]=d^{i+1}(x_j)\in 
\mathfrak{h}$, the space $\mathfrak{g}$ is closed under the commutator bracket $[,]$ in $S$ and hence is a Lie 
subalgebra of $(S,[,])$. Since $d$ is locally nilpotent, $\mathfrak{g}$ is a (finite dimensional) nilpotent Lie algebra 
over $K$ with the Abelian ideal $\mathfrak{h}$ of codimension $1$. By \cite[Theorem 1.1]{CanChristian} $U(\mathfrak{g})$ 
satisfies property $(\diamond)$. The Lie algebra inclusion $\mathfrak{g} \rightarrow R[\theta;d]$ induces an algebra 
map 
$U(\mathfrak{g}) \rightarrow R[\theta; d]$ which is surjective, since $\mathfrak{g}$ contains $\theta$ and all algebra 
generators of $R$. Thus $R[\theta; d]$ also satisfies $(\diamond)$.
\end{proof}

 Proposition \ref{prop_positive} applies in particular to $R=K[x_1,\ldots, x_n]$ and shows that $R[\theta; d]$ satisfies $(\diamond)$ provided $d$ is a locally nilpotent derivation of $K[x_1, \ldots, x_n]$ with $K$ being algebraically closed of characteristic $0$. Examples of locally nilpotent derivations of $R$ 
 are derivations $d$ that satisfy $d(x_1)\in K$ and $d(x_i)\in K[x_1,\ldots, x_{i-1}]$ for all $i>1$.
 
 For an arbitrary non-commutative finitely generated Noetherian domain $R$ however Proposition \ref{prop_positive} might not hold as the following result shows:

\begin{thm}\label{lnd_artinian}  Let $R$ be a Noetherian domain over a field $K$ of characteristic $0$ with locally 
nilpotent derivation $d$ and an element $x\in R$ with $d(x)=1$. Suppose that $S=R[\theta;d]$ satisfies property $(\diamond)$. 
If there exists $a\in R\setminus R^d$ such that $S/S(\theta+a)$ is Artinian as left $S$-module, then $R^d$ is a division ring.
\end{thm}

In order to prove the last Theorem we need an elementary Lemma which will be proven first. 

\begin{lem}\label{lemma_rd} Let $R$ be a domain over a field $K$ of characteristic $0$ and $d$ a locally nilpotent 
derivation such that there exists $x\in R$ with $d(x)=1$. Set $S=R[\theta;d]$.
Then the lattice $\mathcal{L}({_SR})$ of $d$-stable left ideals of $R$ and the lattice $\mathcal{L}(R^d)$ of left ideals 
of $R^d$ are isomorphic. The lattice isomorphism is given by the mutual inverse maps:
$$ I\mapsto I\cap R^d, \qquad \forall I\in \mathcal{L}({_SR}) \qquad \mbox{ and } \qquad J\mapsto RJ, \qquad  \forall 
J\in \mathcal{L}(R^d).$$
\end{lem} 

\begin{proof}
Let $R$ be a domain over a field $K$ of characteristic zero and $d$ a $K$-linear derivation such that there exists $x\in 
R$ with $d(x)=1$. Bavula showed in \cite[Lemma 2.1]{Bavula} that  $R=R^d[x;\delta]$ where $\delta(r)=xr-rx$ for all 
$r\in R^d$. In particular elements of $R$ are polynomials  in $x$ with coefficients $a_i\in R^d$.

Let $I$ be a $d$-stable left ideal of $R$. 
For any element $\gamma=\sum_{i=0}^n x^ia_i \in I$, with $a_i\in R^d$, we have $(n!)^{-1}\theta^n\cdot \gamma = a_n \in 
I\cap R^d$ and therefore $x^na_n=(n!)^{-1}x^n\theta^n\cdot \gamma \in R(I\cap R^d)$. By induction on the degree of elements it follows that 
 $I=R(I\cap R^d)$.

On the other hand if $J$ is a left ideal of $R^d$ and $a \in R^d\cap (RJ)$, then $a = \sum_{i=0}^n x^ia_i\in RJ$ with 
$a_1, \ldots, a_n\in J$. As $R=R^d[x;\delta]$ is a differential operator ring, by comparing the coefficients of $x$ we 
have $a=a_0\in J$. Hence $R^d\cap (RJ)=J$.
\end{proof}

\begin{proof}[Proof of Theorem \ref{lnd_artinian}]
Let $a\in R\setminus R^d$ be any element and consider $I=S(\theta+a)\theta$. We first show that $S/I$ is $\theta$-torsionfree.
Since $S=R\oplus S(\theta+b)$ for any $b\in R$, we have 
$$S = R \oplus S\theta = R\oplus R\theta \oplus I.$$ In particular any element of 
$S/I$ can be uniquely written  as $\gamma=g_0 + g_1\theta + I$ with $g_0,g_1 \in R$. Then 
\begin{equation}\theta\gamma = g_0\theta + d(g_0) + g_1\theta^2 + d(g_1)\theta + I = d(g_0) + (g_0 + d(g_1) - g_1a)\theta + I.\end{equation}
Suppose $\theta\gamma = I$, then $d(g_0)=0$, \textit{i.e.} 
\begin{equation} g_0\in R^d \qquad \mbox{and} \qquad g_0+d(g_1)=g_1a.\end{equation} 
As mentioned before,  $R=R^d[x;\delta]$ where $\delta(r)=xr-rx$ for all $r\in R^d$.   Let us denote the degree  of $g\in R$, considered as a polynomial in $x$, by 
$\mathrm{deg}_x(g)$. The constant polynomials are precisely the elements of $R^d$.
Moreover  $\mathrm{deg}_x(d(g))<\mathrm{deg}_x(g)$ holds for $g\in R\setminus\{0\}$ (where by definition $\mathrm{deg}_x(0) := -\infty$). 
Hence if $g_1$ were not zero, then
\begin{equation}\mathrm{deg}_x(g_0+d(g_1)) \leq \mathrm{deg}_x(g_1)< \mathrm{deg}_x(g_1) + \mathrm{deg}_x(a) = \mathrm{deg}_x(g_1a),\end{equation}
as $R$ is a domain and $a\not\in R^d$. This however would contradict $g_0+d(g_1)=g_1a$. Thus 
$g_1=0$ and $g_0=0$, \textit{i.e.} $\gamma=0$. This shows that $S/I$ is $\theta$-torsionfree. 

Next we will show that $S\theta/I$ is an essential left $S$-submodule of $S/I$. Hence let $U$ be a non-zero 
left $S$-submodule of $S/I$. Choose a non-zero element $\gamma=g_0+g_1\theta + I\in U$ with $\mathrm{deg}_x(g_0)$ 
being minimal.  Suppose $g_0\neq 0$. The fact that $\theta\gamma = d(g_0) + g_1'\theta + I$ is non-zero and belongs to 
$U$ for some $g_1'\in R$ with $\mathrm{deg}_x(d(g_0))<\mathrm{deg}_x(g_0)$ yields a contradiction to the choice of 
$\gamma$. Thus $g_0=0$ and $\gamma=g_1\theta +I \in U\cap S\theta/I$ is a non-zero element; showing that $S\theta/I$ is 
essential in $S/I$.

Suppose now that  $S$ satisfies $(\diamond)$ and that $S/S(\theta+a)$ is Artinian as left $S$-module for some $a\in R\setminus R^d$. Set $I=S(\theta + a)\theta$.  Then $S/I$ is a finitely generated essential extension of the Artinian left $S$-module $S/S(\theta+a)\simeq S\theta/I$. By property $(\diamond)$ $S/I$ is Artinian and so is its factor module $S/S\theta$. 
Since $S/S\theta$ is isomorphic to $R$ as left $S$-module, as mentioned in the introduction, $R$ is an Artinian left $S$-module. By Lemma \ref{lemma_rd}, there exists a lattice isomorphism between the lattice of $d$-stable left ideals of $R$ and left ideals of $R^d$. Hence $R^d$ is left Artinian and therefore  a division ring as it is also a domain.
\end{proof}

Consider  $R=A_1(\mathbb{C})[x]$, which is a Noetherian domain of Krull dimension $2$ that satisfies property $(\diamond)$, because any maximal ideal $\m$ of the centre of $R$ is of  the form $\m=\langle x-\lambda\rangle$, with $\lambda \in \CC$; the quotient ring  $R/\m \simeq A_1(\mathbb{C})$ does  satisfy property $(\diamond)$ and thus  by \cite[Proposition 1.6]{injective_modules_over_down-up_algebras}, $R$ 
satisfies property $(\diamond)$. Let $d$ be the derivation $\frac{\partial}{\partial x}$ of $R$ and set 
$$S=R[y;d]=A_1(\mathbb{C})[x][y;\frac{\partial}{\partial x}],$$ \textit{i.e.} $S\simeq A_2(\mathbb{C})$. 
Then $d$ is locally nilpotent and $d(x)=1$. Stafford showed in  \cite[Theorem 1.1]{nonholonomic_modules_over_weyl_algebras_and_enveloping_algebras} that there  exists an element $a\in R$ such that $S/S(\theta+a)$ is simple. Since $R^d=A_1(\mathbb{C})$ is not a division ring, $S$ 
does not satisfy $(\diamond)$ by Theorem \ref{lnd_artinian}.

%%%%%%%%%

\section{ Commutative Noetherian $d$-primitive rings}
A ring $R$ is called \emph{$d$-simple} if the only $d$-ideals of $R$ are $0$ and $R$.

\begin{lem}\label{Artinian-general} Let $R$ be a commutative Noetherian domain with derivation $d$. Set $S=R[\theta;d]$. Then
 $S/S\theta$ is Artinian as left $S$-module if and only if $R$ is $d$-simple.
\end{lem}
\begin{proof} By the preceding remark, $S/S\theta$ is isomorphic to $R$ as $S$-modules. 
Hence if we suppose $S/S\theta$ to be Artinian as a left $S$-module, then so is $R$ as a left $S$-module, \textit{i.e.} $R$  
satisfies the descending chain condition on $d$-stable left ideals.

Consider a $d$-ideal $I$ of $R$. The descending chain $I\supseteq I^2\supseteq I^3 
\supseteq \cdots $ stops and there exists $n\geq 0$ such that $I^n = I^{2n}$, \textit{i.e.} $I^n$ is an idempotent ideal of $R$. 
As the only idempotent ideals of a commutative Noetherian domain are $0$ and $R$ (see \cite[Corollary 5.4]{Eisenbud}), 
we have $I^n=0$ and hence $I=0$ as $R$ is a domain or $I^n=R$, \textit{i.e.} $I=R$ - showing that $R$ is $d$-simple. On the other 
hand if $R$ is $d$-simple, then $R$ is simple as a left $S$-module and hence $S/S\theta$ is simple and therefore Artinian as a left $S$-module.
\end{proof}

We will show that if $R$ has no $\ZZ$-torsion and is $d$-primitive, \textit{i.e.} $R$ contains a maximal ideal $\m$ which does not contain any non-zero $d$-ideals, then $S=R[\theta;d]$ satisfying $(\diamond)$ will force $R$ to be $d$-simple. 

The following Proposition appears in \cite[Proposition 3.1]{GoodearlWarfield82} where the authors refer to  a paper by Hart \cite{Hart}. However, since in \cite{Hart} the ring $S$ is supposed to be simple, we will include Hart's proof  for the convenience of the reader.

\begin{prop}[Hart]\label{simple-general} Let $R$ be a commutative ring and  $S=R[\theta;d]$ for some derivation $d$ of $R$. 
Suppose that  $\m$ is a maximal ideal of $R$ that is not $d$-stable and satisfies $\mathrm{char}(R/\m)=0$. Then $S\m$ is a maximal left 
ideal of $S$.
\end{prop}

\begin{proof}
Let $I$ be any left ideal of $S$ properly containing $S\m$. Let $f\in I\setminus S\m$ be an element of minimal degree $n$ with leading coefficient $a\not\in \m$. Suppose $n>0$.
For any $m\in \m$ we have $fm\in S\m\subseteq I$, hence $fm-mf \in I$ which has degree less than the one of $f$. By the minimality of the degree of $f$,  $fm-mf \in S\m$.
Note that the leading coefficient of $fm-mf$ belongs to $\m$ and equals $nd(m)a$. Since $R/\m$ is a field of characteristic zero and $a\not\in\m$,  $d(m)\in \m$ for all $m\in \m$. Hence  $d(\m)\subseteq \m$, contradicting our assumption on $\m$. Hence $f=a$ must be a non-zero constant polynomial. Since $a\not\in \m$, it  has an inverse mod $\m$. Thus $I=S$.
\end{proof}

Proposition \ref{simple-general} and Lemma \ref{Artinian-general} show that if $R$ is a commutative Noetherian 
domain that is not $d$-simple and has a maximal ideal $\m$ that is not $d$-stable, then $S/S\m\theta$ is a cyclic left 
$S$-module which is an extension of the simple left $S$-module $S/S\m\simeq S\theta/S\m\theta$ such that $S/S\theta$ is 
not Artinian.  Moreover, this extension is essential as the following Lemma shows:

\begin{lem}\label{Lemma_prime}
 Let $R$ be a commutative ring with derivation $d$  and set $S=R[\theta;d]$.
 If $\pp$ is a prime ideal of $R$ that does not contain any non-zero $d$-ideal, then $S\theta/I$ is an essential left $S$-submodule of $S/I$ where $I=S\pp \theta$.
 \end{lem}

\begin{proof} The aim is to show that any non-zero left $S$-submodule $U$ of $S/I$ contains a non-zero element from $S\theta/I$. 
Suppose there exists a non-zero element  $f+I \in U$, such that the leading and constant coefficients of $f$ do not belong to $\pp$.  If $f$ is a constant, say $f=a$, 
$$(a\theta - d(a))a = a^2 \theta \not\in S\pp\theta = I,$$
as $a\not\in\pp$. Otherwise write $f=a\theta^n + g\theta + b$ with  $g\in S$ an element of degree less than $n-1$ and $a, b \not\in \pp$. 

\begin{eqnarray*}
(b\theta- d(b))f  &=& b\theta a\theta^n - d(b)a\theta^n  + (b\theta - d(b))g\theta + b\theta b - d(b)b \\
 &=& ba\theta^{n+1} + (bd(a) - d(b)a)\theta^n  + (b\theta - d(b))g\theta + b^2\theta +bd(b) - d(b)b \\
 &=& ba\theta^{n+1} + (bd(a) - d(b)a)\theta^n  + ((b\theta - d(b))g+b^2)\theta 
\end{eqnarray*}
Since $a,b \not\in \pp$ and $\pp$ is a prime ideal, $ab\not\in \pp$. Hence 
the leading coefficient of $(b\theta - d(b))f$ does not belong to $\pp$ and we get that $(b\theta - d(b))f \not\in S\pp\theta = I$.
Therefore $$0\neq (b\theta -d(b))f + I \in U\cap (S\theta/I).$$
The rest of the proof consists in proving the existence of such element $f+I$ in $U$.
Let $h+I$ be a non-zero element of $U$ of minimal degree. Let  $h=\sum_{i=0}^n a_i \theta^i$. 
Clearly if $a_0=0$, then $h\in S\theta$, hence $h+I \in U\cap (S\theta/I)$ and we are done.

Let $a_0\neq 0$ and suppose first that $n=0$, \textit{i.e.} $h=a$ is a constant polynomial. Then there exists a least integer $m\geq 0$ such that $d^m(a)\not\in \pp$ as otherwise the elements $d^i(a)$, $i\geq 0$ would generate a non-zero $d$-ideal in $\pp$ contradicting our assumption on $\pp$.
By induction one proves (in case $m>0$):
\begin{equation}
 \theta^m a = \left( \sum_{i=0}^{m-1}\theta^{m-1-i} d^i(a) \right)\theta + d^m(a)
\end{equation}
Since $a, d(a), d^2(a), \ldots, d^{m-1}(a)$ belong to $\pp$, we have that $\theta^m a - d^m(a) \in I$. In other words
\begin{equation}
 f+I := \theta^m (a+I) = d^m(a) + I \in U
\end{equation}
is an element that can be represented by a polynomial whose leading and constant coefficients are not in $\pp$.

Suppose now that the degree of $h$ is  $n>0$. Then $a_n\not\in \pp$ as  otherwise
$\theta^{n-1}a_n\theta \in I$ and $h-\theta^{n-1}a_n\theta$ would be an element of degree less than the degree of $h$,  contradicting our minimality assumption. 
Let  $m\geq 0$ be the least integer such that $d^m(a_0)\not\in \pp$. Write $h=a_n\theta^n + g\theta + a_0$ with $g$ a polynomial of  degree less than $n-1$. Thus 
\begin{equation}
f+I:=  \theta^m h + I = \theta^ma_n\theta^n + \theta^mg\theta + \theta^m a_0 + I = a_n \theta^{m+n} + \tilde{g}\theta + d^{m}(a_0) + I \in U,
\end{equation}
for some polynomial $\tilde{g}$ of degree less than $n+m-1$. Hence $f+I\in U$ is represented by an element whose leading and constant  coefficients do not belong to  $\pp$.
\end{proof}

The next result follows from  Lemmas \ref{Artinian-general}, \ref{Lemma_prime}  and Proposition \ref{simple-general}.

\begin{thm}\label{essentiality_Theorem} Let $R$ be a commutative Noetherian domain with derivation $d$ such that $R$ is not $d$-simple. Suppose that $R$ has a maximal ideal $\m$ that does not contain any non-zero $d$-ideal and such that $\mathrm{char}(R/\m)=0$. Then $S/S\m\theta$ is a non-artinian essential extension of the simple module $S\theta/S\m\theta$ where $S=R[\theta; d]$.
\end{thm}

As a Corollary we obtain the following result:

\begin{thm}\label{essentiality}
Let $R$ be a commutative Noetherian domain such that  $R[\theta;d]$ satisfies $(\diamond)$ for some non-zero derivation $d$.
If $R$ is $d$-primitive and has no $\ZZ$-torsion, then  $R$ is $d$-simple.
\end{thm}

\begin{proof}
By hypothesis $R$ contains a maximal ideal $\m$ which does not contain any non-zero $d$-ideals.
If $\mathrm{char}(R/\m)=p>0$, then $pR\subseteq \m$ is a non-zero $d$-ideal as 
$R$ is supposed to be $\ZZ$-torsionfree. Hence $\mathrm{char}(R/\m)=0$. 
If $S=R[\theta;d]$ satisfies $(\diamond)$, then  $S/S\m\theta$ cannot be a non-artinian essential extension of a simple module. Hence by Theorem \ref{essentiality_Theorem} $R$ must be $d$-simple.
\end{proof}

%By Theorem Proposition \ref{simple-general} $S/S\m$ is a simple left $S$-module, where $S=R[\theta; d]$. As $\m$ is also a prime ideal, we have by Lemma \ref{Lemma_prime} that $S/I$ is an essential extension of $S\theta/I$ where $I=S\m\theta$. By property $(\diamond)$, $S/S\theta$ is Artinian and by Lemma  \ref{Artinian-general}, $R$ is $d$-simple.

\section{Ore extensions of $K[x]$}

We apply the general results of the last section to $R=K[x]$ with $K$ a field of characteristic $0$. 
Denote by $A_1(K)=K[x][\theta; \frac{\partial}{\partial x}]$ the first Weyl algebra over $K$.

\begin{cor}\label{diamond_for_kx}
Let $K$ be a field of characteristic $0$ and $d$ a $K$-linear derivation of $K[x]$.
The following statements are equivalent for $S=K[x][\theta; d]$:
\begin{enumerate}
\item[(a)] $S$ satisfies property $(\diamond)$.
\item[(b)] $d=\lambda\frac{\partial}{\partial x}$ for some $\lambda \in K$.
\item[(c)] $S\simeq K[x,\theta]$ or $S\simeq A_1(K)$.
\item[(d)] $S$ is commutative or has Krull dimension $1$.
\end{enumerate}
\end{cor}

\begin{proof} 
First note, that any $K$-linear derivation $d$ of $R=K[x]$ is of the form $d=f\frac{\partial}{\partial x}$ for some polynomial $f\in R$.

$(a)\Rightarrow (b)$ Suppose  $f$ is not constant. Let $\alpha\in K$ be any element that is not a root of $f$ (which exists as $K$ is infinite). 
Then $\m=\langle  x-\alpha\rangle$ is a maximal ideal which does not contain any non-zero $d$-ideal, because if $I=\langle g\rangle \subseteq \m$ is a non-zero $d$-ideal, then $d(g)=h_1g$ for some $h_1\in R$. Let  $h_2 \in R$ and $n>0$ be such that $g=h_2(x-\alpha)^n$, $(x-\alpha)\nmid h_2$.
 Then
$$ h_1h_2(x-\alpha)^n =  d(g) = d(h_2)(x-\alpha)^n + h_2 f n(x-\alpha)^{n-1}$$
which implies $(h_1h_2-d(h_2))(x-\alpha) = nh_2f$. Since $(x-\alpha) \nmid f$ and $(x-\alpha)\nmid h_2$ we obtain a contradiction.
Hence if $f$ is not constant, then $\m$ does not contain any non-zero $d$-ideal and $R$ is $d$-primitive. Thus if $S$ satisfies property $(\diamond)$, then by Corollary \ref{essentiality}  $R$ is $d$-simple. 
However  $\langle f \rangle$ is a proper, non-zero $d$-ideal of $R$.  Hence $f$ has to be constant, \textit{i.e.} $d=\lambda \frac{\partial}{\partial x}$ for some $\lambda \in K$.

The implication $(b) \Rightarrow (c)$ holds since $S\simeq A_1(K)$ if $\lambda \neq 0$ and $(c)\Rightarrow (d)$ is clear as the Krull dimension of $A_1(K)$ over a field $K$ is $1$.

$(d)\Rightarrow (a)$ Any commutative Noetherian ring satisfies $(\diamond)$ by Matlis' result in \cite{Matlis}. If $S$ is a Noetherian domain of left Krull dimension $1$, then every proper factor module of it is Artinian by \cite[Corollary 1.5]{GordonRobson} or \cite[Theorem 10]{Krause}. Hence any cyclic essential extension of a simple left $S$-module is Artinian, \textit{i.e.} $S$ satisfies $(\diamond)$.
\end{proof}

M. Awami, M. Van den Bergh, F. Van Oystaeyen in \cite[2.1]{AwamiVandenBerghVanOystaeyen} and J. Alev and A. Dumas in 
\cite[Proposition 3.2]{AlevDumas} proved that given any $K$-linear automorphism $\sigma$ of $K[x]$ and $\delta$ a 
$K$-linear  $\sigma$-derivation, the Ore extension $S=K[x][y;\sigma,\delta]$ is isomorphic to one of the following 
algebras: $K[x,y]$,  a quantum plane, a quantum Weyl algebra or to a differential operator ring. Since any automorphism 
of $K[x]$ is such that $\sigma(x)=qx+b$ for some $q,b\in K$, following the proof of \cite {AwamiVandenBerghVanOystaeyen} 
and \cite{AlevDumas}, we have:
\begin{itemize}
 \item[$\bullet$] If $q\neq 1$, then $S\simeq K[x'][y; \sigma', \delta]$ where $x'=x+b(q-1)^{-1}$ and $\sigma'(x')=qx'$. 
Now if $p(x)\in K[x]$ and  $r\in K$ are such that $\delta(x')=p(x')(1-q)x'+r$ then $$(y+p(x'))x'=qx'(y+p(x'))+r.$$
If $r=0$, it is easy to see that $S\simeq K_q[x',y']=K\langle x',y'|y'x'=qx'y'\rangle$ for a suitable change of 
variables.\\
If $r\neq 0$, taking $y''=r^{-1}(y+p(x))$, $S\simeq A_1^q(K):=K\langle x',y''|y''x'=qx'y''+1\rangle$.
 \item[$\bullet$] If $q=1$ and $b=0$, $S$ is either $K[x,y]$ or a differential operator ring,  $S=K[x][y;\delta]$.
 \item[$\bullet$] If $q=1$ and $b\neq 0$, $S\simeq K[x'][y; \sigma', \delta]$ by making $x'=b^{-1}x$, $\sigma'(x')=x'+1$ 
and $\delta'(x')=b^{-1}\delta(x)$. Since $$(y+\delta(x))x'=(x'+1)(y+\delta(x))$$
it follows that $S\simeq K[y'][x'; -y'\frac{\partial}{\partial y'}]$.
\end{itemize}

\begin{thm}\label{OreExtension}
Let  $K$ be a field of characteristic zero and let $\sigma$ be a $K$-linear automorphism of $K[x]$ and $d$ be a 
$K$-linear $\sigma$-derivation of $K[x]$.
Let $q,b\in K$ such that $\sigma(x)=qx+b$. The following statements are equivalent:
\begin{enumerate}
 \item[(a)] $S=K[x][\theta;\sigma, d]$ satisfies property $(\diamond)$;
 \item[(b)] $S\simeq K_q[x,y]$ or $S\simeq A_1^q(K)$ for $q$ a root of unity (including $q=1$);
 \item[(c)] $q\neq 1$ is a root of unity or $q=1$ and $d(x)$ is a constant polynomial;
 \item[(d)] $\sigma\neq id$ has finite order or $\sigma=id$ and $d$ is locally nilpotent. 
 \end{enumerate}
\end{thm}

\begin{proof} By Corollary \ref{diamond_for_kx} any algebra of the form $K[x][\theta;d]$ satisfying $(\diamond)$ has to 
be isomorphic to the first Weyl algebra or to the polynomial ring, \textit{i.e.} to $A_1^q(K)$ or to $K_q[x,y]$ for $q=1$.
If $q$ is not a root of unity, then \cite[Theorem 3.1]{PaulaIan}, respectively 
\cite[Theorem 4.2]{PaulaIan}, shows that $K_q[x,y]$, respectively $A_1^q(K)$, does not satisfy $(\diamond)$. 
On the other hand if $q\neq 1$ is a root of unity, then $K_q[x,y]$ and $A_1^q(K)$ are PI-algebras and hence in 
particular FBN which satisfy $(\diamond)$ 
by Jategaonkar's result (see \cite{Jategaonkar_conj}). The case $q=1$ is obtained from the fact that the first Weyl 
algebra is a prime Noetherian algebra of Krull dimension $1$.
Together with the characterisation above, this shows $(a)\Leftrightarrow (b)\Leftrightarrow (c)$.

$(c)\Leftrightarrow (d)$ Note that for all $n>1$ we have $\sigma^n(x)=q^nx+\frac{q^n-1}{q-1}b$ if $q\neq 1$. Suppose 
$\sigma\neq id$, then $\sigma$ has order $n$ if and only if $q$ is an $n$th root of unity.
\end{proof}

It still remains unclear how to extend Theorem \ref{diamond_for_kx} to polynomial rings 
$R=K[x_1, \ldots, x_n]$ in $n$ variables over a field $K$ of characteristic $0$.

%\begin{acknowledgements}\label{ackref}
\section{Acknowledgment}
The authors would like to thank the referee for his/her careful reading and comments that improved the paper.
This research was funded by the European Regional Development Fund through the programme COMPETE and by the Portuguese 
Government through the FCT 
- Fundação para a Ciência e a Tecnologia under the project PEst-C/MAT/UI0144/2013. The second author was supported by 
the grant SFRH/BD/33696/2009.
%\end{acknowledgements}

%\bibliographystyle{amsplain}
\providecommand{\bysame}{\leavevmode\hbox to3em{\hrulefill}\thinspace}
\providecommand{\MR}{\relax\ifhmode\unskip\space\fi MR }
\providecommand{\MRhref}[2]{%
  \href{http://www.ams.org/mathscinet-getitem?mr=#1}{#2}
}
\providecommand{\href}[2]{#2}

\end{document}